\documentclass[11 pt]{amsart}
\usepackage{amsmath,enumerate,amsthm,amssymb,amscd}

\DeclareMathOperator{\Supp}{Supp}
\newcommand{\gothic}{\mathfrak}
\newcommand{\pd}{\operatorname{pd}}
\newcommand{\p}{{\gothic{p}}}
\newcommand{\q}{{\gothic{q}}}

\newcommand{\Spec}{\operatorname{Spec}}

\newcommand{\Hom}{\operatorname{Hom{}}}
\newcommand{\Ext}{\operatorname{Ext{}}}
\newcommand{\Tor}{\operatorname{Tor{}}}

\newcommand{\depth}{\operatorname{depth}}

\renewcommand{\hat}{\widehat}
\renewcommand{\phi}{\varphi}
\renewcommand{\to}{{\longrightarrow}}

\newcommand{\im}{\operatorname{im}}

\newcommand{\id}{\operatorname{id}}
\newcommand{\fd}{\operatorname{fd}}

\newtheorem{thm}{Theorem}[section]
\newtheorem{cor}[thm]{Corollary}
\newtheorem{prop}[thm]{Proposition}
\newtheorem{lemma}[thm]{Lemma}
\newtheorem{defn}[thm]{Definition}

\numberwithin{equation}{section}

\begin{document}
\title{The acyclicity of the Frobenius functor for modules of finite flat dimension}
\author{Thomas Marley and Marcus Webb}
\address{Department of
Mathematics\\
University of Nebraska-Lincoln\\Lincoln,  NE 68588-0130}
\email{ tmarley1@unl.edu}
\address{Department of
Mathematics\\
University of Nebraska-Lincoln\\Lincoln,  NE 68588-0130}
\email{ mwebb4@math.unl.edu}


\begin{abstract}
Let $R$ be a commutative Noetherian local ring of prime characteristic $p$ and $f:R\to R$ the Frobenius ring homomorphism. For $e\ge 1$ let
$R^{(e)}$ denote the ring $R$ viewed as an $R$-module via $f^e$. Results of Peskine, Szpiro, and Herzog
state that for finitely generated modules $M$,  $M$ has finite projective dimension if and only if $\Tor^R_i(R^{(e)},M)=0$ for all $i>0$ and all (equivalently, infinitely many) $e\ge 1$.  We prove this statement holds for arbitrary modules using the theory of flat covers and minimal flat resolutions.
\end{abstract}

\date{December 29, 2014}

\bibliographystyle{amsplain}

\maketitle

\begin{section}{Introduction}
\end{section}
Let $R$ be a commutative Noetherian ring of prime characteristic $p$  and $f:R\to R$ the Frobenius map.  For $e\ge 1$ let $R^{(e)}$ be the ring $R$ considered as an $R$-module via $f^e$;  i.e., for $r\in R, s\in R^{(e)}$, $r\cdot s:= r^{p^e}s$.  A classic result of Kunz \cite{K} states that $R$ is regular if and only if $R$ is reduced and $R^{(e)}$ is flat as an $R$-module for all (equivalently, some) $e\ge 1$.  Subsequently, Peskine and Szpiro \cite[Th\'eor\`em 1.7]{PS} proved that if $\mathbf{P}$ is a finite projective resolution of a finitely generated $R$-module $M$ then for all $e\ge 1$,  $R^{(e)}\otimes_R \mathbf{P}$ is a projective resolution of $R^{(e)}\otimes_R M$; that is, finitely generated modules of finite projective dimension are acyclic objects with respect to the Frobenius functors $R^{(e)}\otimes_R -$.   A year later Herzog \cite[Satz 3.1]{H} showed the converse holds in the case the Krull dimension of $R$ is finite;  namely, if $M$ is a finitely generated $R$-module
and $\mathbf{P}$ is a projective resolution of $M$ such that $R^{(e)}\otimes_R \mathbf{P}$ is acyclic for infinitely many integers $e$, then $M$ has finite projective dimension.  An interesting question is whether these results hold for arbitrary $R$-modules, not just finitely generated ones.  
In Corollary \ref{frob}(a) and Theorem \ref{converse}, we give an affirmative answer to this question:
\medskip

\begin{thm} \label{mainresult} Let $R$ be a Noetherian ring of prime characteristic and $M$  an $R$-module.  Then the following hold:
\begin{enumerate}[(a)]
\item If $\fd_R M<\infty$ then for all $i,e>0$,  $\Tor_i^R(R^{(e)}, M)=0$ and $\fd_{R^{(e)}} R^{(e)}\otimes_R M=\fd_R M$.
\item If $R$ has finite Krull dimension and $\Tor^R_i(R^{(e)},M)=0$ for all $i>0$ and infinitely many integers $e$, then $\pd_R M<\infty$.
\end{enumerate}
\end{thm}
\noindent  Here $\fd_R M$ and $\pd_RM$ denote the flat and projective dimensions of $M$, respectively.  Note that part (a) implies that if $\pd_R M<\infty$ then $\pd_{R^{(e)}} R^{(e)}\otimes_R M\le \pd_R M$  for all $e$ (although we do not know of any examples where the inequality is strict).
We also prove an analogue of Theorem \ref{mainresult} for injective dimension in the case the Frobenius map is finite (cf. Corollaries \ref{frob}(b) and \ref{dual3}).

A special case of part (a) of Theorem \ref{mainresult} is well-known  and is what inspired this investigation:   Suppose $\underline{x}=x_1,\dots,x_n$ is a regular sequence on $R$ and $\mathbf C=\mathbf C(\underline x; R)$ is the \v Cech complex on $R$ with respect to $x$.  Then $\mathbf C$ is a finite flat resolution of the local cohomology module $H^n_{(\underline{x})}(R)$ and $R^{(e)}\otimes_R \mathbf C\cong \mathbf C(\underline x; R^{(e)})\cong \mathbf C(\underline x^{p^e}; R)$  is acyclic since $x_1^{p^e},\dots, x_n^{p^e}$ is an $R$-sequence.

The proofs of Peskine-Szpiro and Herzog both reduce to the case $R$ is local and utilize minimal projective (free) resolutions.  The minimality condition for such projective resolutions can be expressed by saying that all the differentials are zero when tensored with the residue field.  It can be easily shown that projective resolutions of this type do not necessarily exist for arbitrary modules.   However, flat resolutions with this kind of minimality condition do exist for a large class of modules (i.e., cotorsion modules), which we show is  sufficient to prove Theorem \ref{mainresult}.   
The theory of flat covers, cotorsion modules,  and minimal flat resolutions, as developed by Enochs and Xu in \cite{E} and \cite{EX}, are essential ingredients in all our arguments.   In Section 2, we summarize some basic properties of these notions as well as prove some auxiliary results which are used in later sections.    Minimal flat resolutions have many properties analagous to those of minimal injective resolutions.  In particular, the flat modules appearing in a minimal flat resolution of a module $M$ are uniquely determined (up to isomorphism) by invariants  which we call the {\it Enochs-Xu} numbers of $M$.  The Enochs-Xu numbers of a module are in some sense the dual of the Bass numbers of a module.   In Section 3, we prove a vanishing result for Enochs-Xu numbers which will be crucial in our proof of part (a) of Theorem \ref{mainresult}.

Sections 4 and 5 are devoted to the proofs of parts (a) and (b), respectively, of Theorem \ref{mainresult}.   For both parts it is sufficient to consider the case when $R$ is a local ring, in which case every module of finite flat dimension has finite projective dimension by the Jensen-Raynaud-Gruson theorem (\cite[Proposition 6]{J} and \cite[Seconde partie, Th\'eor\`eme 3.2.6]{RG}).  Our strategy is  to apply the methods of
Peskine-Szpiro and Herzog to minimal flat resolutions.   There are several difficulties which arise:  the finitistic flat dimension may exceed the depth of the ring (cf. \cite[Corollary 5.3]{B});  minimal flat resolutions do not in general localize; and the modules appearing in minimal flat resolutions are not generally finitely generated.   We are able to overcome these difficulties by, in part, proving that the depths of the modules in degrees exceeding $\depth R$ in finite minimal flat resolutions are infinite, and showing that, in the case the Frobenius map is finite, the acyclicity of Frobenius for minimal flat resolutions of cotorsion modules commutes with the colocalization functor $\Hom_R (R_p, -)$.   And while the flat modules in question are not typically finitely generated, we are able to reduce to the case where they are completions of free $R$-modules.

\medskip

\noindent {\bf Acknowledgments:}  The authors would like to thank  Edgar Enochs for his help in our understanding of flat covers.  The authors also express their gratitude to everyone in the UNL algebra seminar, in particular Lucho Avramov, Peder Thompson, Mark Walker, and Wenliang Zhang,  for their useful feedback on this project.

\begin{section}{Flat covers and cotorsion modules}
\end{section}

In this section we collect the results we will regarding flat covers, cotorsion modules, and minimal flat resolutions.  We refer the reader to \cite{X} for background on this material.

\newpage

\begin{defn}{\rm  Let  $M$ be an $R$-module.  An $R$-homomorphism $\phi:F\to M$ is called a {\it flat cover} of $M$ if the following hold:
\begin{enumerate}[(a)] 
\item $F$ is flat;
\item for every map $\psi:G\to M$ with $G$ flat, there exists a homomorphism $g:G\to F$ such that $\psi=\phi g$; and 
\item if $h:F\to F$ satisfies $\phi=\phi h$ then $h$ is an isomorphism.  
\end{enumerate}}
\end{defn}
\noindent
By abuse of language, we sometimes refer to the module $F$ as the `flat cover' of $M$, rather than the homomorphism $\phi$.  In \cite{BEE}, it is proved that flat covers exists for all modules over all rings.  However, in this work we will only need the existence of flat covers for modules over commutative Noetherian rings of finite Krull dimension, which was proved in \cite{X2}.
It is easily seen that flat covers are surjective (since every module is a homomorphic image of a flat module), that a flat cover of a module is unique up to isomorphism,  and that the flat cover of a flat module is an isomorphism.  

\begin{defn}{\rm An $R$-module $M$ is called {\it cotorsion} if $\Ext^1_R(F,M)=0$ for every flat $R$-module $F$.}
\end{defn}
\noindent
It is easily seen that if $M$ is cotorsion then $\Ext^i_R(F,M)=0$ for all $i\ge 1$ and all flat $R$-modules $F$.  By \cite[Lemma 2.2 and Corollary]{E}, the kernel of a flat cover is cotorsion and a flat cover of a cotorsion module is cotorsion.  For an $R$-module $M$ we let $C_R(M)$ denote the cotorsion envelope of $M$, which exists by \cite[Theorem 3.4.6]{X}.  The cotorsion envelope of a flat module is flat by \cite[Theorem 3.4.2]{X}.  If $F$ is flat and cotorsion then $F\cong \prod_{\p\in \Spec R} T(p)$ where each $T(p)$ is the completion with respect to the $pR_p$-adic topology of a free $R_p$-module $G(p)$;   furthermore, the ranks of the free $R_p$-modules $G(p)$ are uniquely determined by $F$ (\cite[Theorem]{E}).  For each $p\in \Spec R$, let $\pi(p,F)$ denote the rank (possibly infinite) of $G(p)$.  

Given an $R$-module $M$, a {\it minimal flat resolution} of $M$ is a complex 
$$\mathbf{F}: \qquad \cdots \to F_i\xrightarrow{\partial_i} F_{i-1}\to \cdots \to F_0\to 0$$
such that $\operatorname{H}_i(\mathbf F)=0$ for $i>0$, $\operatorname{H}_0(\mathbf F)\cong M$, and for each $i$, the natural map $F_i\to \operatorname{coker} \partial_{i+1}$ is a flat cover.  It is clear that every $R$-module has a minimal flat resolution, and that any two minimal flat resolutions of $M$ are (chain) isomorphic.  Since the flat cover of a flat module is an isomorphism, it follows that if $\fd_RM=n<\infty$ then the length of any minimal flat resolution of $M$ is $n$.   Note that, by the remarks in the preceding paragraph, if ${\mathbf F}$ is a minimal flat resolution of $M$ then $F_i$ is cotorsion for all $i\ge 1$.  For each $i\ge 1$ and $p\in \Spec R$
we set $\pi_i(p,M):= \pi(p, F_i)$.   For $i=0$ and $p\in \Spec R$, we set $\pi_0(p,M):=\pi(p, C_R(F_0))$.  Note that if $M$ is cotorsion, so is $F_0$ and thus $\pi_0(p,M)=\pi(p,F_0)$ for all $p\in \Spec R$.  We call the the invariants $\pi_i(p,M)$ the {\it Enochs -Xu} numbers of $M$.   Note that $\fd_R M\le n$ if and only if $\pi_i(p,M)=0$ for all $p\in \Spec R$ and $i>n$.

We now state a theorem of Enochs and Xu:

\begin{thm}\label{EX}(\cite[Theorem 2.1 and 2.2]{EX}) Let $R$ be a Noetherian ring.  Then for any $R$-module $M$ we have:
\begin{enumerate}[(a)]
\item $\pi_i(p,M)=\pi_i(p,C_R(M))$ for all $i\ge 0$ and $p\in \Spec R$.
\item If $M$ is cotorsion then $\pi_i(p,M)=\dim_{k(p)} \Tor^{R_p}_i(k(p), \Hom_R(R_p, M))$ for all $i\ge 0$ and $p\in \Spec R$.
\end{enumerate}
\end{thm}

In particular, note that part (a) of this theorem implies that $\fd_RM=\fd_R C_R(M)$.

We now establish some additional results on cotorsion modules which are needed in sections 3 and 4.  We thank Edgar Enochs for showing us a proof of part (b) of the following lemma, which is implicit in \cite{EX}:

\begin{lemma} \label{cotorsion1}  Let $R$ be a Noetherian ring and $S$ a flat $R$-algebra.
\begin{enumerate}[(a)]
\item If $C$ is a cotorsion $R$-module and $T$ a flat $S$-module then $\Hom_R(T,C)$ is a cotorsion $S$-module.
\item If $F$ is a flat and cotorsion $R$-module and $p\in \Spec R$ then $\Hom_R(R_p,F)$ is a flat and cotorsion $R_p$-module.
\end{enumerate}
\end{lemma}
\begin{proof}
To prove (a), let $A$ be a flat $S$-module. As $S$ is flat over $R$,   any flat $S$-module is also flat as an $R$-module.   Thus,  $\Ext^i_R(P\otimes_S A, C)=0$ for all $i>0$ and all projective $S$-modules $P$.   Thus, we have a Grothendieck spectral sequence
$\Ext^p_R(\Tor_q^S(T,A), C)\Rightarrow \Ext^{p+q}_{S}(A, \Hom_R(T,C))$ (cf.  \cite[Theorem 11.54]{R}).
The spectral sequence collapses, giving  $\Ext^i_S(A, \Hom_R(T,C))\cong \Ext^i_R(T\otimes_S A, C)=0$ for $i>0$, as $T\otimes_S A$ is a flat $S$-module and therefore flat as an $R$-module as well.
Hence, $\Hom_R(T,C)$ is a cotorsion $S$-module.

For (b), as $F$ is flat and cotorsion we have $F\cong \prod_{\q\in \Spec R} T(q)$ where each $T(q)$ is the completion of a free $R_q$-module.  We'll show that $\Hom_R(R_p, F)\cong \prod_{q\subseteq p}T(q)$, which is flat and cotorsion by \cite[Theorem]{E}.
For each $q\in \Spec R$ let $\rho_q:F\to T(q)$ be the projection map.  Let $\phi\in \Hom_R(R_p, F)$.  Then, as $R_p$ is divisible by each element in $R\setminus p$, the same is true for the image $(\rho_q\phi)(R_p)$ for all $q$.   Suppose $q\not\subseteq p$.  Let $r\in q\setminus p$.
Then for all $n\ge 1$ we have $(\rho_q\phi)(R_p)\subseteq r^n(\rho_q\phi)(R_p)\subseteq q^nT(q)$.  As $T(q)$ is separated in the $qR_q$-adic topology, we conclude that $(\rho_q\phi)(R_p)=0$.  Thus, $\Hom_R (R_p, F)\cong \Hom_R (R_p, G)$, where $G=\prod_{q\subseteq p} T(q)$.
But since each $T(q)$ for $q\subseteq p$ is an $R_p$-module, $G$ is an $R_p$-module.  Thus, $\Hom_R(R_p, F)\cong G$.
\end{proof}

\begin{lemma} \label{cotorsion2} Let $R$ be a Noetherian ring of finite Krull dimension and $M$ an $R$-module. Suppose $M$ has a (left) cotorsion resolution $\mathbf C$;  i.e., there exists an exact sequence 
$$\cdots \to C_i\xrightarrow{\phi_i} C_{i-1}\to \cdots \to C_0\to M\to 0$$
where each $C_i$ is cotorsion.  Then
\begin{enumerate}[(a)]
\item $M$ is cotorsion.
\item For any flat $R$-module $T$, $\Hom_R(T,\mathbf C)$ is a cotorsion resolution of $\Hom_R(T,M)$.
\end{enumerate}
\end{lemma}
\begin{proof} For each integer $i\ge 0$ let $K_i=\operatorname{coker} \phi_{i+1}$.  Let $T$ be a flat $R$-module.  Since
$\Ext^i_R(T, C_j)=0$ for all $i\ge 1$ and $j\ge 0$, we have that $\Ext^1_R(T,M)\cong \Ext^{i+1}_R(T, K_i)$ for all $i\ge 0$.  Let $n=\dim R$.
Then $\Ext^{n+1}_R(T, K_n)=0$ as $\pd_R T\le n$  by the Jensen-Raynaud-Gruson theorem (\cite[Proposition 6]{J} and \cite[Seconde partie, Th\'eor\`eme 3.2.6]{RG}).   Thus, $\Ext^1_R(T,M)=0$ and $M$ is cotorsion.   This proves (a).  

For part (b), first note that $\Hom_R(T,C_i)$ is cotorsion for all $i$ by Lemma \ref{cotorsion1}(a).  Let $K_i=\operatorname{coker}\phi_{i+1}$ as above.   Since each $K_i$ has a cotorsion resolution, $K_i$ is cotorsion by part (a).   Applying $\Hom_R(T,-)$ to the exact sequences $0\to K_{i+1}\to C_i\to K_i\to 0$, we obtain that
$$0\to \Hom_R(T, K_{i+1})\to \Hom_R(T, C_i)\to \Hom_R(T, K_i)\to 0$$
is exact for all $i$.  Splicing these short exact sequences, we see that $\Hom_R(T, \mathbf C)$ is a cotorsion resolution of $\Hom_R(T,M)$. 
\end{proof}

We remark that for any $R$-modules $L$, $M$, and $N$ there is a natural $R$-module homomorphism
$$L\otimes_R \Hom_R(M,N)\xrightarrow{\psi} \Hom_R (M, L\otimes_R N)$$
such that for $\ell\in L$, $f\in \Hom_R(M,N)$ and $m\in M$,  $\psi(\ell\otimes f)(m)=\ell\otimes f(m)$.

\begin{lemma} \label{iso} Let $R$ be a Noetherian ring of finite Krull dimension.  Let $M$, $T$, and $F$ be $R$-modules such that $M$ is finitely generated, $T$ is flat, and $F$ is flat and cotorsion.  Then the natural map $\psi: M\otimes_R \Hom_R(T, F)\to \Hom_R(T, M\otimes_R F)$
is an isomorphism.
\end{lemma}
\begin{proof} We first note that the lemma is clear in the case $M$ is a finitely generated free $R$-module.  Let $\mathbf{G}$ be a free resolution of $M$ consisting of finitely generated free $R$-modules.   As $F$ is flat and cotorsion, $\mathbf{G}\otimes_R F$ is a cotorsion resolution of $M\otimes_R F$.  By Lemma \ref{cotorsion2}(b), we obtain that $\Hom_R(T, \mathbf{G}\otimes_R F)$ is a cotorsion resolution of $\Hom_R(T, M\otimes_R F)$.  Now consider the commutative diagram:
$$
\begin{CD}
G_1\otimes_R \Hom_R(T,F)@>>> G_0\otimes_R \Hom_R(T,F) @>>> M\otimes_R \Hom_R(T,F)@>>> 0 \\
@VV\cong V @VV\cong V @VVV @. \\
\Hom_R(T, G_1\otimes_R F)@>>>  \Hom_R(T, G_0\otimes_R F)@>>> \Hom_R(T, M\otimes_R F)@>>>0 \\
\end{CD}
$$
where the rows are exact and the vertical arrows are the natural maps.  Since the first two vertical maps are isomorphisms, so is the third by the Five Lemma.
\end{proof}

We will need the following colocalization result for Tor:

\begin{prop} \label{localTor} Let $R$ be a Noetherian ring of finite Krull dimension.  Let $L$ be a finitely generated $R$-module and $M$ a cotorsion $R$-module.   Suppose $\Tor_i^{R}(L,M)=0$ for all $i>0$. Then $\Tor^{R_p}_i(L_p,\Hom_R(R_p,M))=0$ for all $p\in \Spec R$ and $i>0$.
\end{prop}
\begin{proof} Let $p\in \Spec R$.  It suffices to prove that $\Tor_i^R(L, \Hom_R(R_p,M))=0$ for all $i>0$.  Let $\mathbf{F}$ be a minimal flat resolution of $M$.  As $M$ is cotorsion, so is $F_i$ for all $i$.  As $\Tor_i^R(L,M)=0$ for all $i>0$ we have that $L\otimes_R \mathbf{F}$ is acyclic.   We remark that $L\otimes_R F_i$ is cotorsion for all $i$.  To see this, let $\mathbf{G}$ be a free resolution of $L$ by finitely generated free $R$-modules.  Then  $\mathbf{G}\otimes_R {F_i}$ is a cotorsion resolution of $L\otimes_R F_i$.  Hence, $L\otimes_R F_i$ is cotorsion by Lemma \ref{cotorsion2}(a).  By Lemma \ref{cotorsion2}(b), $\Hom_R(R_p, L\otimes_R \mathbf{F})$
is acyclic.  By Lemma \ref{iso}, $L\otimes_R \Hom_R(R_p, \mathbf{F})\cong \Hom_R(R_p, L\otimes_R \mathbf{F})$ as complexes.
Thus, $L\otimes_R\Hom_R(R_p, \mathbf F)$ is acyclic.  However, $\Hom_R(R_p, \mathbf F)$ is a flat resolution of $\Hom_R(R_p, M)$
by Lemmas \ref{cotorsion1}(a) and \ref{cotorsion2}(b).   Thus, $\Tor_i^R(L,\Hom_R(R_p, M))\cong H_i(L\otimes_R \Hom_R(R_p, \mathbf{F}))=0$ for $i>0$.
\end{proof}

\begin{section}{A vanishing result for Bass numbers and Enochs-Xu numbers}
\end{section}

Recall that for an $R$-module $M$ and $p\in \Spec R$, the Bass number  $\mu_i(p,R)$ is defined to be $\dim_{k(p)}\Ext^i_{R_p}(k(p), M_p)$, where $k(p)$ is the field $R_p/pR_p$.  If $\mathbf{I}$ is a minimal injective resolution of $M$, then $\mu_i(p,M)$ is the cardinality of the number of copies of $E_R(R/p)$ in any decomposition of $I^i$ into indecomposable modules.  It is well-known that if $M$ is finitely generated and $\id_R M<\infty$ then $\mu_i(p,M)=0$ for all $p\in \Spec R$ and  $i>\depth R_p$.  The following proposition shows this result still holds even if $M$ is not finitely generated.
\medskip

\begin{prop} \label{bass} Let $R$ be a Noetherian ring and $M$ an $R$-module such that $\id_RM<\infty$.
Then $\mu_i(p,M)=0$ for all $p\in \Spec R$ and $i>\depth R_p$.
\end{prop}
\begin{proof} By localizing at $p$, we can assume $R$ is local with maximal ideal $m$ and that $p=m$. Let $\mu_i(M):=\mu_i(m,M)$,  $s:=\depth R$, and $k:=R/m$.  Certainly $\mu_i(M)=0$ for $i>\id_RM$.  Suppose $r>s$ and  $\mu_i(M)=0$ for all $i>r$.  It suffices to show that $\mu_r(M)=0$.  Let $x_1,\dots,x_s\in m$ be a regular sequence and $I=(x_1,\dots,x_s)$.   Since $\depth R/I=0$ we have an exact sequence
$$0\to k\to R/I\to R/J\to 0$$ where $J=(I,y)$ for some $y\in m$.   Since $\pd_R R/I=s$, $\Ext^i_R(R/I,M)=0$ for $i>s$.   From the long exact sequence on Ext, we have $\Ext^i_R(k,M)\cong \Ext^{i+1}_R(R/J,M)$ for $i>s$.  As $r>s$ and $\mu_i(M)=0$ for all $i>r$,  we obtain $\Ext^i_R(R/J,M)=0$ for $i>r+1$.  Let $n:=\id_R M$, $\mathbf I$ a minimal injective resolution of $M$ and $\mathbf L:=\Hom_R(R/J, \mathbf I)$.  Then $\operatorname{H}^i(\mathbf L)=\Ext^i_R(R/J,M)=0$ for $i>r+1$, giving us an exact sequence
$$
\begin{CD}
0@>>> K @>>>L^{r+1}@> \partial^{r+1} >> L^{r+2} @>>> \cdots @>>> L^n\to 0,
\end{CD}
$$
where $K=\ker \partial^{r+1}$. Now, for each $i\ge 0$ we have
\begin{align*}
L^i&=\Hom_R(R/J,I^i) \\
&\cong \Hom_R(R/J, \oplus  E_R(R/p)^{\mu_i(p,M)}) \\
&\cong \bigoplus_{p\in V(J)} E_{R/J}(R/p)^{\mu_i(p,M)}.
\end{align*}
Note that for $p\neq m$, $\Tor^R_j(k, E_{R/J}(R/p))=0$ for all $j$ since $E_{R/J}(R/p)$ is an $R_p$-module. Since $\mu_i(M)=0$ for $i\ge r+1$, we see that $\Tor^R_j(k, L_i)=0$ for all $j$ and all $i\ge r+1$.
Hence, from the long exact sequence above, we see that $\Tor_j^R(k,K)=0$ for all $j$.  Now, we have an exact sequence
$$K\to \Ext^{r+1}_R(R/J,M)\to 0.$$
Tensoring this sequence with $k$ and noting that $\Ext^{r+1}_R(R/J,M)\cong \Ext^r_R(k,M)$, we conclude that $k\otimes_R \Ext^r_R(k,M)=0$.  Hence, $\mu_r(M)=0$.

\end{proof}

\begin{cor} \label{dual} Let $R$ be a Noetherian ring and $M$ an $R$-module such that $\fd_RM<\infty$.  
Then $\Tor_i^{R_p}(k(p), M_p)=0$ for all $p\in \Spec R$ and $i>\depth R_p$.
\end{cor}
\begin{proof}  By localizing at $p$, we may assume $R$ is local and $p=m$.  Let $k=R/m$, $E=E_R(k)$,
and let $(-)^{\text{v}}=\Hom_R(-,E)$ be  the Matlis dual functor.  Since the Matlis dual is exact and $T^{\text{v}}$ is injective for any flat $R$-module $T$, we see that $\id_R M^{\text{v}}<\infty$.  By Proposition \ref{bass} and \cite[Corollary 11.54]{R}, we
have $\Tor^R_i(k,M)^{\text{v}}\cong \Ext^i_R(k,M^{\text{v}})=0$ for $i>\depth R$.  Hence, $\Tor_i^R(k,M)=0$ for $i>\depth R$. 

\end{proof}

We now prove the analogue of Proposition \ref{bass} for the Enochs-Xu numbers of a module of finite flat dimension:

\begin{prop} \label{EXnumbers} Let $R$ be a Noetherian ring and $M$ an $R$-module such that $\fd_R M<\infty$.
Then $\pi_i(p,M)=0$ for all $p\in \Spec R$ and $i>\depth R_p$.
\end{prop}
\begin{proof} By Theorem \ref{EX}(a),  we may assume $M$ is a cotorsion module.  Let $p\in \Spec R$.  Note that $\Hom_R(R_p,M)$ is a cotorsion $R_p$-module by Lemma \ref{cotorsion1}(a) and $\fd_{R_p} \Hom_R(R_p,M)<\infty$ by Lemmas \ref{cotorsion1}(b) and \ref{cotorsion2}(b). Using part (b) of Theorem \ref{EX},  we have that $\pi_i(p,M)=\pi_i(pR_p, \Hom_R(R_p,M))$ for all $i\ge 0$.  
Resetting notation, we can now assume $R$ is a Noetherian local ring and $p=m$, the maximal ideal of $R$.   By Theorem \ref{EX}(b) and Corollary \ref{dual}, we have $\pi_i(m,M)=\dim_{R/m}\Tor^R_i(R/m, M)=0$ for $i>\depth R$.  
\end{proof}

\begin{section}{An acyclicity theorem for finite flat resolutions}
\end{section}

We begin by making a convention regarding depth for (possibly) non-finitely generated modules.  Let $R$ be a Noetherian local ring with maximal ideal $m$ and $M$ an arbitrary $R$-module.  For $i\ge 0$, let  $H^i_m(M)$ denote the $i$th local cohomology module with support in $m$.  We define the {\it depth} of $M$ by
$$\depth M:=\inf \{i\in \mathbb N_0\mid H^i_m(M)\neq 0\}.$$
Note that under this definition,  $\depth M=\infty$ if and only if $H^i_m(M)=0$ for all $i$.
We remark that if $F$ is a flat $R$-module then $H^i_m(F)\cong H^i_m(R)\otimes_R F$; hence, $\depth F\ge \depth R$.  We'll need the following result concerning the depths of certain cotorsion flat modules:
\medskip

\begin{lemma} \label{lem} Let $\phi:R\to S$ be a homomorphism of Noetherian local rings.  Let $m$ and $n$ be the maximal ideals of $R$ and $S$, respectively, and assume $\phi(m)\subseteq n$.   Let $F$ be a cotorsion flat $R$-module such that $\pi(m,F)=0$.  Then for every $S$-module $N$ we have $\depth_S N\otimes_R F=\infty$. In particular, $\depth F=\infty$.
\end{lemma}
\begin{proof} We first remark that as $F$ is a flat $R$-module, $H^i_n(N\otimes_R F)\cong H^i_n(N)\otimes_R F$ for all $i$. Note that, as $mS\subseteq n$, $\Supp_R H^i_n(N)\subseteq \{m\}$ for all $i$.  Thus, it suffices to prove that given any $R$-module $M$ with $\Supp_R M\subseteq \{m\}$ then $M\otimes_R F=0$.   As $F$ is flat and cotorsion, we have $F\cong \prod_{p\in \Spec R}T(p)$ where $T(p)$ is the $pR_p$-adic completion of a free $R_p$-module of rank $\pi(p, F)$.  
As $\pi(m,F)=0$, we can write this decomposition as $F\cong \prod_{p\neq m} T(p)$.   Let $M$ be an $R$-module of finite length.   As $R$ is Noetherian, $M$ is finitely presented and hence $M\otimes_R F\cong \prod_{p\neq m} (M\otimes_R T(p))=0,$ since $M\otimes_R R_p=0$ for all $p\neq m$.  Suppose now that $M$ is an arbitrary $R$-module such that $\Supp_R M\subseteq \{m\}$.  Then $M$ is the direct limit of the direct system $\{M_{\alpha}\}$ consisting of the finite length $R$-submodules of $M$ (with morphisms the inclusion maps).  Thus
$$M\otimes_R F\cong (\varinjlim M_{\alpha})\otimes_R F\cong \varinjlim (M_{\alpha}\otimes_R F)=0.$$
\end{proof}

Our next result is the acyclicity lemma of Peskine and Szpiro \cite{PS}, but stated for complexes whose modules are not necessarily finitely generated.  The proof is largely {\it mutatis mutandis}, but we include it for the convenience of the reader.

\begin{prop} \label{ps} (\cite[Lemme d'acyclicit\'e 1.8]{PS}) Let $R$ be a Noetherian local ring and consider a bounded complex $\mathbf T$ of $R$-modules
$$\mathbf T:\qquad 0\to T_s\xrightarrow{f_s} T_{s-1}\to \cdots \xrightarrow{f_0} T_0\to 0.$$
Suppose the following two conditions hold for each $i>0$:
\begin{enumerate}
\item $\depth T_i\ge i$; 
\item $\depth \operatorname{H}_i(\mathbf T)=0$ or $\operatorname{H}_i(\mathbf T)=0$.
\end{enumerate}
Then $\operatorname{H}_i(\mathbf T)=0$ for all $i>0$.
\end{prop}
\begin{proof}  If $s=0$ there is nothing to show, so we assume $s>0$.  Since $\operatorname{H}_s(\mathbf T)$ is a submodule of $T_s$ and $\depth T_s\ge s\ge 1$, we see that $\depth \operatorname{H}_s(\mathbf T)>0$.  Thus, $\operatorname{H}_s(\mathbf T)=0$.
Suppose for some $j\ge 1$ we have that $\operatorname{H}_i(\mathbf T)=0$ for $j< i\le s$.  It suffices to show that $\operatorname{H}_j(\mathbf T)=0$.  Let $C_j=\im f_{j+1}$.  Then the truncated complex
$$0\to T_s\xrightarrow{f_s} T_{s-1}\to \cdots \xrightarrow{f_{j+2}} T_{j+1}\to C_j\to 0$$
is exact.  Using the inequalities $\depth T_i\ge i$ and induction, one can show that $\depth C_j\ge j+1$.  Let $K_j=\ker f_j$ and consider the exact sequence  $$0\to C_j\to K_j\to \operatorname{H}_j(\mathbf T)\to 0.$$   Since $K_j\subseteq T_j$ and $\depth T_j\ge j\ge 1$,
we conclude that $\depth K_j\ge 1$.   Since $\depth C_j\ge 2$, we see that $\depth \operatorname{H}_j(\mathbf T)\ge 1$.  Thus, $\operatorname{H}_j(\mathbf T)=0$.
\end{proof}

We now prove the main result of this section, which generalizes \cite[Corollary 1.10]{PS} to modules of finite flat dimension:

\begin{thm}\label{main} Let $\phi:R\to S$ be a homomorphism of Noetherian rings such that  for every $q\in \Spec S$ and $p=\phi^{-1}(q)$ one has that $\depth S_q\ge \depth R_p$.   Let $M$ be an $R$-module of finite flat dimension.  Then 
\begin{enumerate} 
\item[(a)] $\Tor_i^R(S,M)=0$ for all $i>0$.
\item[(b)] $\fd_S S\otimes_R M\le \fd_R M$.
\item[(c)] If $k(p)\otimes_R S\neq 0$ for all $p\in \Spec R$, then $\fd_S S\otimes_R M = \fd_R M$.
\end{enumerate}
\end{thm}
\begin{proof} We prove part (a) by way of contradiction.  Suppose $\Tor_i^R(S,M)\neq 0$ for some $i\ge 1$.  Let $q\in \Spec S$ be a prime minimal with respect to the property that $\Tor_i^R(S,M)_q\neq 0$ for some $i\ge 1$ and let $p:=\phi^{-1}(q)$.   By replacing $R$, $S$, and $M$ with $R_p$, $S_q$, and $M_p$, we may assume $(R,m)$ and $(S,n)$ are local rings,  $\phi(m)\subseteq n$,  $\Tor_i^R(S,M)\neq 0$ for some $i\ge 1$, and  $\Supp_S\Tor_i^R(S,M)\subseteq \{n\}$ for all $i\ge 1$.   Let
$\mathbf{F}$ be a
minimal flat resolution of $M$ and $\mathbf{L}:=S\otimes_R {\mathbf F}$.    Since $\operatorname{H}_i(\mathbf L) \cong \Tor^R_i(S,M)$  and $\Supp_S \Tor^R_i(S,M)\subseteq \{n\}$ for all $i\ge 1$, we have that $\depth \operatorname{H}_i(\mathbf L)=0$ or $\operatorname{H}_i(\mathbf L)=0$ for all $i\ge 1$.  

We claim that $\depth_S L_i\ge i$ for all $i$.  Since $L_i$ is a flat $S$-module, $\depth_S L_i\ge \depth S$.  Suppose $i>\depth S$.   Then $F_i$ is a cotorsion flat $R$-module and $\pi(m,F_i)=\pi_i(m, M)=0$ by Proposition \ref{EXnumbers} since $i>\depth R$. Then $\depth_S L_i=\infty$ by Lemma \ref{lem}.  This proves the claim.  By Proposition \ref{ps}, we obtain that $\operatorname{H}_i(\mathbf L)=\Tor_i^R(S,M)=0$ for all $i\ge 1$, which is a contradiction.  This proves part (a).

To see part (b), let $\mathbf F$ be a flat resolution of $M$ of length $n=\fd_R M$.  Then $S\otimes_R \mathbf F$ is an $S$-flat resolution of $S\otimes_R M$, since $\Tor^R_i(S,M)=0$ for all $i>0$ by part (a).   Hence, $\fd_S S\otimes_R M\le n$.

To prove part (c) it suffices by part (b) to show that $\fd_S S\otimes_R M\ge \fd_R M$.   Let $n=\fd_R M$ and $J$ an ideal maximal with respect to $\Tor_n^R(R/J,M)\neq 0$.   Then by \cite[Proposition 2.2]{AB}, $J=p$ is prime and $\Tor_n^R(R/p, M)$ is a (nonzero) $k(p)$-module, where $k(p)=R_p/pR_p$.   Let $(\mathbf F, \mathbf \partial)$ be a flat resolution of $M$ of length $n$.    Then we have an exact sequence
$$0\to \Tor^R_n(R/p,M)\to k(p)\otimes_R F_n\xrightarrow{1\otimes \partial_n} k(p)\otimes_R F_{n-1}.$$
Now, tensoring with $S$ over $R$ (which is the same as tensoring by $S\otimes_R k(p)$ over $k(p)$, which is flat over $k(p)$), we have
$$0\to S\otimes_R \Tor^R_n(R/p, M)\to S\otimes_R k(p)\otimes_R F_n\xrightarrow{1\otimes 1\otimes \partial_n} S\otimes_R k(p)\otimes_R F_{n-1}$$
is exact.  Reassociating, we obtain an exact sequence
$$0\to S\otimes_R \Tor^R_n(R/p, M)\to (k(p)\otimes_R S)\otimes_S (S\otimes_R F_n)\xrightarrow{(1\otimes 1)\otimes (1\otimes \partial_n)} (k(p)\otimes_R S)\otimes_S (S\otimes_R F_{n-1}).$$  Since $S\otimes_R \mathbf{F}$ is a flat $S$-resolution of $S\otimes_R M$, this sequence shows that 
$$\Tor^S_n(k(p)\otimes_R S, S\otimes_R M)\cong S\otimes_R \Tor^R_n(R/p, M)\cong S\otimes_R k(p)^{\ell},$$
which is nonzero, since $l\neq 0$ and $S\otimes_R k(p)\neq 0$.  Hence, $\fd_S S\otimes_R M\ge n$.
\end{proof}

In the case the ring homomorphism $\phi:R\to S$ is finite, we have the following dual result for modules of finite injective dimension:

\begin{cor}\label{dual2} Let $\phi:R\to S$ be as in in Theorem \ref{main} and assume $S$ is a finitely generated $R$-module.  Let $M$ be an $R$-module of finite injective dimension.  Then 
\begin{enumerate} 
\item[(a)] $\Ext_R^i(S,M)=0$ for all $i>0$.
\item[(b)] $\id_S \Hom_R(S,M)\le \id_R M$.
\item[(b)] If $k(p)\otimes_R S\neq 0$ for all $p\in \Spec R$, then $\id_S \Hom_R(S,M) = \id_R M$.
\end{enumerate}
\end{cor}
\begin{proof} For part (a), it suffices to prove the statement locally at every prime ideal (as $S$ is a f.g. $R$-module).  So assume $R$ is local and let $(-)^{\text{v}}$ denote the Matlis dual functor for $R$.  By \cite[Theorem 11.57]{R}, $\Tor_i^R(A, M^{\text{v}})\cong \Ext^i_R(A,M)^{\text{v}}$ for all $i\ge 0$ and all finitely generated $R$-modules $A$.  In particular, we have that $\fd_R M^{\text{v}}=\id_R M<\infty$.  Applying part (a) of Theorem \ref{main}, we have $\Ext^i_R(S,M)^{\text{v}}\cong \Tor^R_i(S,M^{\text{v}})=0$ for all $i>0$.  Thus,
$\Ext^i_R(S,M)=0$ for all $i>0$.  

As in Theorem \ref{main}, part (b) follows readily from part (a).  

To prove (c), it suffices to show that $\id_S \Hom_R (S, M)\ge \id_R M$.  Let $n=\id_R M$.  Then there exists $p\in \Spec R$
such that $\Ext^n_{R_p}(k(p),M_p)\neq 0$.  An argument analagous to the one in the proof of part (c) of Theorem \ref{main} shows that 
$$\Ext^n_{S_p}(k(p)\otimes_R S, \Hom_R(S,M)_p)\cong \Hom_{R_p}(S\otimes_R k(p), \Ext^n_{R_p}(k(p), M_p)),$$
which is nonzero since $k(p)\otimes_R S\neq  0$ and $\Ext^n_{R_p}(k(p),M_p)\neq 0$.  Thus, $\id_S \Hom_R(S,M)\ge n$.

\end{proof}

We now apply our results to the Frobenius map.  The following corollary generalizes Th\'eor\`ems 1.7 and 4.15 of \cite{PS}, which were proved for finitely generated modules.

\begin{cor} \label{frob} Let $R$ be a Noetherian ring of prime characteristic, $M$ an $R$-module, and $e\ge 1$ an integer. 
\begin{enumerate}
\item[(a)] If $\fd_R M<\infty$ then  $\Tor_i^R(R^{(e)},M)=0$ for all $i>0$ and $\fd_{R^{(e)}} R^{(e)}\otimes_R M=\fd_R M$.
\item[(b)] If the Frobenius map of $R$ is finite and $\id_R M<\infty$ then $\Ext^i_R(R^{(e)},M)=0$ for all $i>0$ and $\id_{R^{(e)}} \Hom_R(R^{(e)},M)=\id_R M$.
\end{enumerate}
\end{cor}
\begin{proof}  It suffices to prove both (a) and (b) in the case $e=1$.  If $f:R\to R$ is the Frobenius map, then for all $q\in \Spec R$,  $f^{-1}(q)=q$ and $k(q)\otimes_R R^{(e)}\neq 0$  for all $e\ge 1$. The conclusions now follow from Theorem \ref{main} and Corollary \ref{dual2}.
\end{proof}

\begin{section}{Proof of the converse}
\end{section}

In this section we prove part (b) of Theorem \ref{mainresult}.  Before doing so, we need to establish some results
 on completions of free modules.   Let $(R,m)$ be a Noetherian local ring and $M$ an $R$-module.  
We let $\hat{M}$ denote the $m$-adic completion of $M$; i.e., $\hat{M}:=\varprojlim M/m^nM$.    If $M$ is separated (i.e., $\cap_n m^nM=0$), then the natural map $M\to \hat{M}$ is injective.  We note that for any $n\ge 1$, $M/m^nM\cong \hat{M/m^nM}\cong \hat{M}/\hat{m^nM}$.  In particular, if $M$ is separated, then $\hat{m^nM}\cap M=m^nM$.  Finally, as $m^n$ is finitely generated, we have that
$m^n\hat{M}=\hat{m^nM}$.
\smallskip

\begin{lemma} \label{completion} Let $(R,m)$ be a Noetherian local ring, $F$ a free $R$-module, and $N$ a finitely generated $R$-module.  Let $t:=t(R)$ be the least integer such that $m^t\cap H^0_m(R)=0$.  (Such a $t$ exists by the Artin-Rees lemma.) Then the following hold:
\begin{enumerate}[(a)]
\item $H^0_m(\hat{F})=H^0_m(F)$.
\item For all $n\ge t$,   $m^n\hat{F}\cap H^0_m(\hat{F})=0$.
\item $\hat{F}\otimes_R N\cong  \hat{F\otimes_R N}$.
\end{enumerate}
\end{lemma}
\begin{proof} Let $X$ be a basis for $F$.  Then $F\cong \bigoplus_{\alpha\in X} R_{\alpha}$, where $R_{\alpha}=R$ for all $\alpha$.
We write this as $F=\oplus_{\alpha} R$ for short.  Note that $m^nF=\oplus_{\alpha} m^n$ for all $n\ge 1$ and $H^0_m(F)=\oplus_{\alpha} H^0_m(R)$.  In particular, $F$ is separated.  We also observe that the topology on $H^0_m(F)$ induced from the $m$-adic topology on $F$ coincides with the $m$-adic topology on $H^0_m(F)$. (One can check this on each component.) 
Thus, $\hat{F/H^0_m(F)}\cong \hat{F}/\hat{H^0_m(F)}$.    Clearly, $\hat{H^0_m(F)}=H^0_m(F)$ and $H^0_m(F)\subseteq H^0_m(\hat{F})$.  Therefore, to prove (a) it suffices to show that $H^0_m(\hat{F/H^0_m(F)})=0$.   If $R=H^0_m(R)$, this is clear. Otherwise,
we can replace $R$ by $R/H^0_m(R)$ and assume $H^0_m(R)=0$.  In this case, $H^0_m(F)=0$ and it suffices to prove that $H^0_m(\hat{F})=0$.  Let $x\in m$ be a non-zero-divisor on $R$.  Then 
$0\to F\xrightarrow{\mu_x} F$ is exact, where $\mu_x$ is multiplication by $x$.  We claim that $0\to \hat{F}\xrightarrow{\mu_x} \hat{F}$ is exact.  To see this, it suffices to show that the topology induced by $\mu_x^{-1}(m^nF)$ on $F$ coincides with the $m$-adic topology on $F$.   As $x$ is a non-zero-divisor
on $R$, there exists (by the Artin-Rees lemma) an integer $t$ such that $(m^n:_Rx)\subseteq m^{n-t}$ for all $n\ge t$.  Then
$\mu_x^{-1}(m^nF)=\oplus_{\alpha} (m^n:_Rx)\subseteq m^{n-t}F$ for all $n\ge t$.  Hence the topologies coincide, and $x$ is a non-zero-divisor on $\hat{F}$.    Thus, $H^0_m(\hat{F})=0$ and part (a) is proved.

Using part (a) and that $m^n\hat{F}\cap {F}=m^nF$, we have for all $n\ge t$
\begin{align*}
m^n\hat{F}\cap H^0_m(\hat{F})&=m^nF\cap H^0_m(F)\\ &=\oplus_{\alpha} (m^n\cap H^0_m(R))\\ &=0.
\end{align*}
Thus, (b) is proved.

We now prove part (c).  Let $R^r\xrightarrow{\phi} R^s\to N\to 0$ be a presentation for $N$ .   Let $K=\im \phi$.  By the Artin-Rees lemma, there exists an integer $t\ge 1$ such that $m^nR^s\cap K\subseteq m^{n-t}K$ for all $n\ge t$.   As $F$ is free, we have
$0\to F\otimes_R K\to F\otimes_R R^s\to F\otimes_R N\to 0$ is exact.   Then for $n\ge t$,
\begin{align*} 
m^n(F\otimes_R R^s)\cap (F\otimes_R K)&= (\oplus_{\alpha} m^nR^s)\cap \oplus_{\alpha}K \\
&=\oplus_{\alpha}(m^nR^s\cap K )\\
&\subseteq \oplus_{\alpha} m^{n-t}K\\
&\subseteq m^{n-t}(F\otimes_R K).
\end{align*}
Thus, the induced and $m$-adic topologies on $F\otimes_R K$ coincide.  Therefore,
$$0\to \hat{F\otimes_R K}\to \hat{F\otimes_R R^s}\to \hat{F\otimes_R N}\to 0$$
is exact.  Composing with the surjection $\hat{F\otimes_R R^r}\xrightarrow{\hat{1\otimes \phi}}\hat{F\otimes_R K}$ and noting that $\hat{A}\otimes_R R^n$ is naturally isomorphic to $\hat{A^n}$ for any $R$-module $A$ and any positive integer $n$, we obtain the  commutative diagram:
$$
\begin{CD}
\hat{F\otimes_R R^r} @>\hat{1\otimes \phi}>>\hat{F\otimes_R R^s} @>>> \hat{F\otimes_R N}@>>> 0 \\
@VV\cong V @VV\cong V @. @. \\
\hat{F}\otimes_R R^r@>1\otimes \phi>>  \hat{F}\otimes_R R^s @>>> \hat{F}\otimes_R N@>>>0. \\
\end{CD}
$$
Hence, $\hat{F\otimes_R N}\cong \hat{F}\otimes_R N$ by the Five Lemma.
\end{proof}

Before proving part (b) of Theorem \ref{mainresult}, we introduce some notation used in the proof.  For $e\ge 1$ we let $m_e$ denote the maximal ideal of $R^{(e)}$.  For $x\in R$, we let $x_e$ denote the element $x$ considered as an element in $R^{(e)}$.  Thus, $xR^{(e)}=x_e^{p^e}R^{(e)}$.  Given an $R$-module $N$ we let $N^{(e)}$ denote the $R^{(e)}$-module $R^{(e)}\otimes_R N$.  If $\psi:L\to N$ is an $R$-module homomorphism, $\psi^{(e)}$ denotes the map $1\otimes_R \psi: L^{(e)}\to N^{(e)}$. 

\begin{thm} \label{converse}  Let $R$ be a Noetherian ring of prime characteristic $p$ and $M$ an $R$-module.  Assume $R$ has finite Krull dimension.  Suppose $\Tor_i^R(R^{(e)}, M)=0$ for all $i>0$ and infinitely many integers $e$.  Then $\pd_R M<\infty$.
\end{thm}
\begin{proof} We first note that it suffices to prove $\fd_R M<\infty$ by the Jensen-Raynaud-Gruson theorem.    Let $d:=\dim R$.  As $\fd_RM<\infty$ if and only if $\Tor^{R_q}_{d+1}(M_q,N_q)=0$ for every $R$-module $N$ and $q\in \Spec R$, we can assume $R$ is local.  By a standard argument, there exists a faithfully flat local $R$-algebra $(T,n)$ such that $T^{(e)}$ is f.g. as a $T$-module for all $e$ (e.g, \cite[Section 3]{K}).  Note that $\fd_R M=\fd_T (T\otimes_R M)$ and $\Tor_i^T(T^{(e)}, T\otimes_R M)\cong T^{(e)}\otimes_{R^{(e)}} \Tor_i^R(R^{(e)},M)$ for all $i$ and $e$.  Hence, by replacing $R$ by $T$, we may assume $R^{(e)}$ is a finitely generated $R$-module for all $e$.

Let $\phi:F\to M$ be a flat cover and $K=\ker \phi$.  By \cite[Lemma 2.2]{E},  $K$ is cotorsion.  Note that $\fd_R K<\infty$ if and only if $\fd_R M<\infty$ and $\Tor_i^R(R^{(e)}, K)\cong \Tor_{i+1}^R(R^{(e)}, M)$ for all $i\ge 1$ and all $e$. Hence, we may assume that $M$ is cotorsion.

We assume that $\fd_R M=\infty$ and derive a contradiction.  Choose $q\in \Spec R$ minimal with respect to the property that $\fd_{R_q} \Hom_{R}(R_q,M)=\infty$.   By Propostion \ref{localTor}, if $\Tor_i^R(R^{(e)}, M)=0$ for some $e$ and all $i>0$, then $\Tor_i^{R_q}(R_q^{(e)}, \Hom_R(R_q, M))=0$ for all $i>0$.  Note also that $\Hom_R(R_q,M)$ is a cotorsion $R_q$-module by Lemma \ref{cotorsion1}(a) and that $R_q^{(e)}$ is finitely generated as an $R_q$-module for all $e\ge 1$.
Thus, by replacing $R$ with $R_q$ and $M$ with $\Hom_R(R_q,M)$,  we may assume $\fd_R M=\infty$ and $\fd_{R_q} \Hom_R(R_q, M)<\infty$ for all $q\neq m$.

Let $\mathbf{F}$ be a minimal flat resolution of $M$ and let $\phi_i:F_{i}\to F_{i-1}$ denote the differentials of $\mathbf F$.  Since $M$ is cotorsion, so is $F_i$ for all $i$.    By the proof of \cite[Theorem 2.2]{EX}, $\phi_i\otimes_R 1_{R/m}=0$ for all $i$; i.e., $\phi_i(F_i)\subseteq mF_{i-1}$ for all $i$.  Then $\phi_i^{(e)}(F_i^{(e)})\subseteq m_e^{[p^e]}F_{i-1}^{(e)}$ for all $i$ and $e$.

Let $s:=\depth R$ and $\mathbf x\in m$ a regular sequence on $R$ of length $s$.   Let $t:=t(R/(\mathbf x))$ as defined in Lemma \ref{completion}, and let $e$ be an integer such that $m^{[p^e]}\subseteq m^t$ and  $\Tor_i^R(R^{(e)}, M)=0$ for all $i>0$.  Let $S$ denote the $R^{(e)}$-algebra $R^{(e)}/(\mathbf x_e)$ and $n=m_eS$, the maximal ideal of $S$. 
Then $H^0_n(S)\neq 0$ and $n^t\cap H^0_n(S)=0$ by definition of $t$.  Since $\pd_{R^{(e)}} S=s$, we have that  $\Tor_i^{R^{(e)}}(S, M^{(e)})=0$ for all $i>s$.  As $\mathbf{F}^{(e)}$ is an $R^{(e)}$-flat resolution of $M^{(e)}$, we obtain that
$H_i(S\otimes_{R^{(e)}} \mathbf F^{(e)})=0$ for all $i>s$.  For each $i\ge 0$ let $C_i=\im \phi_{i+1}$.  Then for all $i\ge 1$ we have an exact sequence of $R^{(e)}$-modules
$$0\to C_{i}^{(e)}\to F_{i}^{(e)}\to C_{i-1}^{(e)}\to 0.$$
From the remarks above, we have that $C_i^{(e)}=\phi_{i+1}^{(e)}(F_{i+1}^{(e)})\subseteq m_e^{[p^e]}F_i^{(e)}\subseteq m_e^tF_i^{(e)}$
for all $i$.   Since $H_i(S\otimes_R \mathbf F^{(e)})=0$ for all $i>s$,  we have that
$$0\to S\otimes_{R^{(e)}} C_{i}^{(e)}\to S\otimes_{R^{(e)}} F_{i}^{(e)}\to  S\otimes_{R^{(e)}} C_{i-1}^{(e)}\to 0$$
is exact for all $i>s-1$.  Note that $S\otimes_{R^{(e)}} C_i^{(e)}\subseteq n^t(S\otimes_{R^{(e)}} F_i^{(e)})$ for all $i>s-1$.

By our assumptions, we have that  $\fd_{R_q} \Hom_R (R_q, M)\le d-1$ for all $q\neq m$, where $d=\dim R$.
By Theorem \ref{EX}(b),  this implies that
$\pi(q, F_i)=\pi_i(q, M)=\pi_i(qR_q, \Hom_R(R_q,M))=0$ for all $i\ge d$ and all $q\neq m$.  Consequently, for $i\ge d$, $F_i$ is the completion of a free $R$-module $G_i$ of rank $r_i:=\pi_i(m,M)$.  By Lemma \ref{completion}(c), for $i\ge d$ we have that  $F_i^{(e)}\cong R^{(e)}\otimes_R \hat{G_i}\cong \hat{G_i^{(e)}}$ and $S\otimes_{R^{(e)}} F_i^{(e)}\cong \widehat{S\otimes_{R^{(e)}} G_i^{(e)}}$, which is the $n$-adic completion of a free $S$-module of rank $r_i$.
Note that since $\fd_RM=\infty$, $r_i\neq 0$ for all $i\ge d$.   For all $i\ge d-1$, let $B_i=S\otimes_{R^{(e)}} C_i^{(e)}$ and $V_i=S\otimes_{R^{(e)}} F_i^{(e)}$.  From above, for $i\ge d$ we have exact sequences of $S$-modules 
$$0\to B_i \to V_i \to B_{i-1}\to 0$$
where $V_i$ is the completion of a  (nonzero) free $S$-module and $B_i\subseteq n^tV_i$.  In particular, by Lemma \ref{completion}(a), $H^0_n(V_i)\neq 0$ for all $i\ge d$.  By Lemma \ref{completion}(b), $H^0_n(B_i)\subseteq n^tV_i\cap H^0_n(V_i)=0$ for all $i\ge d$.  This implies that $H^0_n(V_i)=0$ for $i\ge d+1$, a contradiction.
\end{proof}

In the case the Frobenius map is finite, we obtain the converse to part (b) of Corollary \ref{frob}.  This generalizes \cite[Satz 5.2]{H}, which was proved in the case the module $M$ is finitely generated.

\begin{cor} \label{dual3} Let $R$ be a Noetherian ring of prime characteristic.  Assume that $R$ has finite Krull dimension and that the Frobenius map is finite.  Let $M$ be an $R$-module and suppose that $\Ext^i_R(R^{(e)}, M)=0$ for all $i>0$ and for infinitely many integers $e$.  Then $\id_RM<\infty$.
\end{cor}
\begin{proof} Since $R^{(e)}$ is a finitely generated $R$-module for all $e$ and $\dim R<\infty$, we may assume $R$ is local.  Let $(-)^{\text{v}}$ be the Matlis dual functor.  Then, as in the proof of Corollary \ref{dual2}(a), $\id_R M=\fd_R M^{\text{v}}$.  By \cite[Theorem 11.57]{R}, $\Tor_i^R(R^{(e)}, M^{\text{v}})\cong \Ext^i_R(R^{(e)},M)^{\text{v}}=0$ for all $i> 0$ and infinitely many $e$.
By Theorem \ref{converse}, $\id_RM= \fd_RM^{\text{v}}<\infty$.
\end{proof}

\end{document}